\documentclass[a4paper,11pt]{article}
\usepackage{graphicx}
\usepackage{amssymb}
\usepackage[french, english]{babel}
\usepackage[utf8]{inputenc}
\usepackage{amsmath,amsfonts,amssymb,amsthm}
\usepackage[T1]{fontenc}
\usepackage{fullpage}
\usepackage{hyperref}
\usepackage{mathrsfs}
\usepackage{relsize}
\usepackage{tikz}
\usepackage{bbm}
\usepackage{appendix}

\newcommand{\eps}{\varepsilon}

\newcommand{\E}{\mathbb E}

\renewcommand{\P}{\mathbb P}

\theoremstyle{definition}

\newtheorem{thm}{Theorem}

\newtheorem{defn}{Definition}
\newtheorem{rem}[defn]{Remark}

\newtheorem{prop}[defn]{Proposition}
\newtheorem{corr}[defn]{Corollary}

\newtheorem{lem}[defn]{Lemma}
\newtheorem{conj}[defn]{Conjecture}

\tikzstyle{every node}=[circle, draw, fill=black!50, inner sep=0pt, minimum width=4pt]
\tikzstyle{rouge}=[circle, draw, fill=red, inner sep=0pt, minimum width=6pt]
\tikzstyle{bleu}=[circle, draw, fill=blue, inner sep=0pt, minimum width=6pt]
\tikzstyle{petitrouge}=[circle, draw, fill=red, inner sep=0pt, minimum width=4pt]
\tikzstyle{petitbleu}=[circle, draw, fill=blue, inner sep=0pt, minimum width=4pt]
\tikzstyle{texte}=[draw=none, fill=none]

\title{\bf{On the mixing time of the flip walk on triangulations of the sphere}}
\author{Thomas \bsc{Budzinski} \footnote{ENS Paris and Université Paris-Saclay, \url{thomas.budzinski@ens.fr}}}
\date{}

\begin{document}

\maketitle

\begin{abstract}
A simple way to sample a uniform triangulation of the sphere with a fixed number $n$ of vertices is a Monte-Carlo method: we start from an arbitrary triangulation and flip repeatedly a uniformly chosen edge. We give a lower bound of order $n^{5/4}$ on the mixing time of this Markov chain.
\end{abstract}

\section{Introduction}

Much attention has been given recently to the study of large uniform triangulations of the sphere. Historically, these triangulations have been first considered by physicists as a discrete model for quantum gravity. Before the introduction of more direct tools (bijection with trees or peeling process), the first simulations \cite{JKP86, KKM85} were made using a Monte-Carlo method based on flips of triangulations.

More precisely, for all $n \geq 3$, let $\mathscr{T}_n$ be the set of rooted type-I triangulations of the sphere with $n$ vertices (that is, triangulations that may contain loops and multiple edges, equipped with a distinguished oriented edge). If $t$ is a triangulation we write $V(t)$ for the set of its vertices and $E(t)$ for the set of its edges. If $t \in \mathscr{T}_n$ and $e\in E(t)$, we write $\mathfrak{flip}(t,e)$ for the triangulation obtained by removing the edge $e$ from $t$ and drawing the other diagonal of the face of degree $4$ that appears. We say that $\mathfrak{flip}(t,e)$ is obtained from $t$ by \textit{flipping} the edge $e$ (cf. Figure \ref{figureflip}). Note that it is possible to flip a loop and to flip the root edge. The only case in which an edge cannot be flipped is if both of its sides are adjacent to the same face like the edge $e_2$ on Figure \ref{figureflip}. In this case $\mathfrak{flip}(t,e)=t$. Note that there is a natural bijection between $E(t)$ and $E \left( \mathfrak{flip}(t,e) \right)$. When there is no ambiguity, we shall sometimes treat an element of one of these two sets as if it belonged to the other.

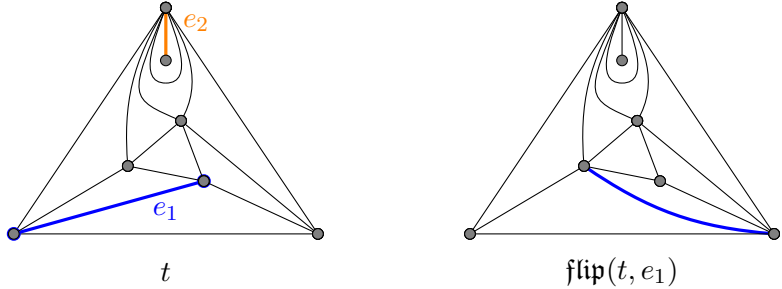
\begin{figure}
\begin{center}
\begin{tikzpicture}
\draw[very thick, orange] (2,3)--(2,2.3);
\draw (2.4,2.8)[orange] node[texte] {$e_2$};
\draw (2,2.3) node{};
\draw (0,0) node{}--(4,0) node{};
\draw (0,0) node{}--(2,3) node{};
\draw (2,3) node{}--(4,0) node{};
\draw (2,3) node{} to[in=180, out=250] (2,2);
\draw (2,3) node{} to[in=0, out=290] (2,2);
\draw (2,3) node{} .. controls (1.5,1.8) .. (2.2,1.5) node{};
\draw (2,3) node{} to[in=60, out=300] (2.2,1.5)  node{};
\draw[very thick, blue] (0,0) node{}--(2.5,0.7) node{};
\draw (2,0.3)[blue] node[texte] {$e_1$};
\draw (4,0) node{}--(2.5,0.7) node{};
\draw (4,0) node{}--(2.2,1.5) node{};
\draw (1.5,0.9) node{}--(2.5,0.7) node{};
\draw (2.5,0.7) node{}--(2.2,1.5) node{};
\draw (1.5,0.9) node{}--(0,0) node{};
\draw (1.5,0.9) node{} to[bend left=18] (2,3) node{};
\draw (1.5,0.9) node{}--(2.2,1.5) node{};
\draw (2,-0.5) node[texte]{$t$};

\begin{scope}[shift={(6,0)}]
\draw (2,3)--(2,2.3);
\draw[very thick, blue] (1.5,0.9) to[bend right=15] (4,0);
\draw (2,2.3) node{};
\draw (0,0) node{}--(4,0) node{};
\draw (0,0) node{}--(2,3) node{};
\draw (2,3) node{}--(4,0) node{};
\draw (2,3) node{} to[in=180, out=250] (2,2);
\draw (2,3) node{} to[in=0, out=290] (2,2);
\draw (2,3) node{} .. controls (1.5,1.8) .. (2.2,1.5) node{};
\draw (2,3) node{} to[in=60, out=300] (2.2,1.5)  node{};
\draw (4,0) node{}--(2.5,0.7) node{};
\draw (4,0) node{}--(2.2,1.5) node{};
\draw (1.5,0.9) node{}--(2.5,0.7) node{};
\draw (2.5,0.7) node{}--(2.2,1.5) node{};
\draw (1.5,0.9) node{}--(0,0) node{};
\draw (1.5,0.9) node{} to[bend left=18] (2,3) node{};
\draw (1.5,0.9) node{}--(2.2,1.5) node{};
\draw (2,-0.5) node[texte]{$\mathfrak{flip}(t,e_1)$};
\end{scope}
\end{tikzpicture}
\end{center}
\vspace{-1cm}
\caption{An example of flip of an edge. The orange edge $e_2$ is not flippable.} \label{figureflip}
\end{figure}

The graph of triangulations of the sphere in which two triangulations are related if one can pass from one to the other by flipping an edge has already been studied in the type-III setting (that is, triangulations with neither loops nor multiple edges): it is connected \cite{W36} and its diameter is linear in $n$ \cite{K97}. We extend these results to our setup in Lemma \ref{irreducibility}. 

We define a Markov chain $(T_n(k))_{k \geq 0}$ on $\mathscr{T}_n$ as follows: conditionally on $(T_n(0), \dots, T_n(k))$, let $e_k$ be a uniformly chosen edge of $T_n(k)$. We take $T_{n}(k+1)=\mathfrak{flip}(T_n(k),e_k)$. It is easy to see that the uniform measure on $\mathscr{T}_n$ is reversible, thus stationary for $\left( T_n(k) \right)_{k \geq 0}$, so this Markov chain will converge to the uniform distribution (the irreducibility is guaranteed by the connectedness results described above and the aperiodicity by the possible existence of non flippable edges). It is then natural to estimate the mixing time of $\left( T_n(k) \right)_{k \geq 0}$ (see Chapter 4.5 of \cite{LPW09} for a proper definition of the mixing time). Our theorem provides a lower bound.

\begin{thm} \label{mainthm}
There is a constant $c>0$ such that for all $n \geq 3$ the mixing time of the Markov chain $(T_n(k))_{k \geq 0}$ is at least $c n^{5/4}$.
\end{thm}

Mixing times for other types of flip chains have also been investigated. For triangulations of a convex $n$-gon without inner vertices it is known that the mixing time is polynomial and at least of order $n^{3/2}$ (see \cite{MT97, MRS98}). In particular, our proof was partly inspired by the proof of the lower bound in \cite{MRS98}. Finally, see \cite{CMSS15} for estimates on the mixing time of the flip walk on \textit{lattice triangulations}, that is, triangulations whose vertices are points on a lattice and with Boltzmann weights depending on the total length of their edges.

The strategy of our proof is as follows: we start with two independent uniform triangulations with a boundary of length $1$ and $\frac{n}{2}$ inner vertices and glue them together along their boundaries. We obtain a triangulation of the sphere with a cycle of length $1$ such that half of the vertices lie on each side of this cycle. We then start our Markov chain from this triangulation and discover one of the two sides of the cycle gradually by a peeling procedure. By using the estimates of Curien and Le Gall \cite{CLGpeeling} and a result of Krikun about separating cycles in the UIPT \cite{Kri04}, we show that after $o(n^{5/4})$ flips, with high probability, the triangulation still has a cycle of length $o(n^{1/4})$, on each side of which lie a proportion at least $\frac{1}{4}$ of the vertices. But by a result of Le Gall and Paulin \cite{LGP08}, this is not the case in a uniform triangulation (this is the discrete counterpart of the homeomorphicity of the Brownian map to the sphere), which shows that a time $o(n^{5/4})$ is not enough to approach the uniform distribution.

\paragraph{Acknowledgements:} I thank Nicolas Curien for carefully reading earlier versions of this manuscript. I also thank the anonymous referee for his useful comments. I acknowledge the support of ANR Liouville (ANR-15-CE40-0013) and ANR GRAAL (ANR-14-CE25-0014).

\section{Combinatorial preliminaries and couplings}

For all $n \geq 3$, we recall that $\mathscr{T}_n$ is the set of rooted type-I triangulations of the sphere with $n$ vertices. For $n \geq 0$ and $p \geq 1$ we also write $\mathscr{T}_{n,p}$ for the set of triangulations with a boundary of length $p$ and $n$ inner vertices, that is, planar maps with $n+p$ vertices in which all faces are triangles except one called the \textit{outer face} whose boundary is a simple cycle of length $p$, equipped with a root edge such that the outer face touches the root edge on its right. We will sometimes refer to $n$ and $p$ as the \textit{volume} and the \textit{perimeter} of the triangulation.

The number of triangulations with fixed volume and perimeter can be computed by a result of Krikun. Here is a special case of the main theorem of \cite{Kri07} (the full theorem deals with triangulations with $r+1$ boundaries but we only use the case $r=0$):
\begin{equation}\label{enumeration}
\# \mathscr{T}_{n,p}=\frac{p(2p)!}{(p!)^2} \frac{4^{n-1} (2p+3n-5)!!}{n! (2p+n-1)!!} \underset{n \to +\infty}{\sim} C(p) \lambda_c^{-n} n^{-5/2},
\end{equation}
where $\lambda_c=\frac{1}{12 \sqrt{3}}$ and $C(p) = \frac{3^{p-2} p (2p)!}{4 \sqrt{2 \pi} (p!)^2}$. In particular, a triangulation of the sphere with $n$ vertices is equivalent after a root transformation to a triangulation with a boundary of length $1$ and $n-1$ inner vertices (more precisely we need to duplicate the root edge, add a loop inbetween and root the map at this new loop, see for example Figure 2 in \cite{CLGmodif}), so
\begin{equation} \label{enumerationSphere}
 \# \mathscr{T}_n = \# \mathscr{T}_{n-1,1} = 2 \frac{4^{n-2} \, (3n-6)!!}{(n-1)! \, n!!}.
\end{equation}

For $n \geq 0$ and $p \geq 1$ we write $T_{n,p}$ for a uniform triangulation with a boundary of length $p$ and $n$ inner vertices, and $T_n$ for a uniform triangulation of the sphere with $n$ vertices. We also recall that the UIPT, that we write $T_{\infty}$, is an infinite rooted planar triangulation whose distribution is characterized by the following equality. For any rooted triangulation $t$ with a hole of perimeter $p$,
\begin{equation} \label{UIPT}
\P \left( t \subset T_{\infty} \right)=C(p) \lambda_c^{|t|},
\end{equation}
where $\lambda_c$ and the $C(p)$ are as above, $|t|$ is the total number of vertices of $t$ and by $t \subset T_{\infty}$ we mean that $T_{\infty}$ can be obtained by filling the hole of $t$ with an infinite triangulation with a boundary of length $p$.

In what follows we will use several times peeling explorations of random triangulations, see section 4.1 of \cite{CLGpeeling} for a general definition. Let $t$ be a triangulation and $\mathscr{A}$ be a peeling algorithm, that is, a way to assign to every finite triangulation with one hole an edge on the boundary of the hole. We write $t^{\mathscr{A}}_j(t)$ for the part of $t$ discovered after $j$ steps of filled-in peeling following algorithm $\mathscr{A}$. By "filled-in" we mean that everytime the peeled face separates the unknown part of the map in two connected components we reveal the one with fewer vertices (if the two components have the same number of vertices we reveal one component picked deterministically). If the map is infinite and one-ended, we reveal the bounded component.

From the enumeration formulas it is possible to deduce precise coupling results between finite and infinite maps. The result we will need is similar to Proposition 12 of \cite{HBP} but a bit more general since it deals with triangulations with a boundary. We recall that in a triangulation $t$ of the sphere or the plane, the \textit{ball} of radius $r$, that we write $B_r(t)$, is the triangulation with holes formed by those faces adjacent to at least one vertex lying at distance at most $r-1$ from the root, along with all their edges and vertices. If $t$ is infinite, the \textit{hull} of radius $r$, that we write $B_r^{\bullet}(t)$, is the union of $B_r (t)$ and all the bounded connected components of its complement. If $t$ is finite, it is the union of $B_r (t)$ and all the connected components of its complement except the one that contains the most vertices (if there is a tie, we pick deterministically a component among those which contain the most vertices). If $T$ is a triangulation with a boundary, we adopt the same definitions but we replace the distance to the root by the distance to the boundary.

\begin{lem}\label{couplagebord}
Let $p_n=o(\sqrt{n})$ and $r_n=o(n^{1/4})$ with $p_n=o(r_n^2)$. Then there are $r'_n=o(r_n)$ and couplings between $T_{n, p_n}$ and $T_{\infty}$ such that
\[ \mathbb{P} \left( B^{\bullet}_{r_n} (T_{\infty}) \backslash B^{\bullet}_{r'_n} (T_{\infty}) \subset B^{\bullet}_{r_n} (T_{n,p_n}) \right) \xrightarrow[n \to +\infty]{} 1. \]
\end{lem}

The above lemma follows from the following. There is a cycle $\gamma'$ of length $p_n$ around the root of $T_{\infty}$ that lies inside of its hull of radius $r'_n$ and a cycle $\gamma$ in $T_{n,p_n}$ that stays at distance at most $r_n$ from its boundary, such that the part of the hull of radius $r_n$ of $T_{\infty}$ that lies outside of $\gamma'$ is isomorphic to the part of $T_{n,p_n}$ that lies between its boundary and $\gamma$ (see Figure \ref{figurecoupling}).

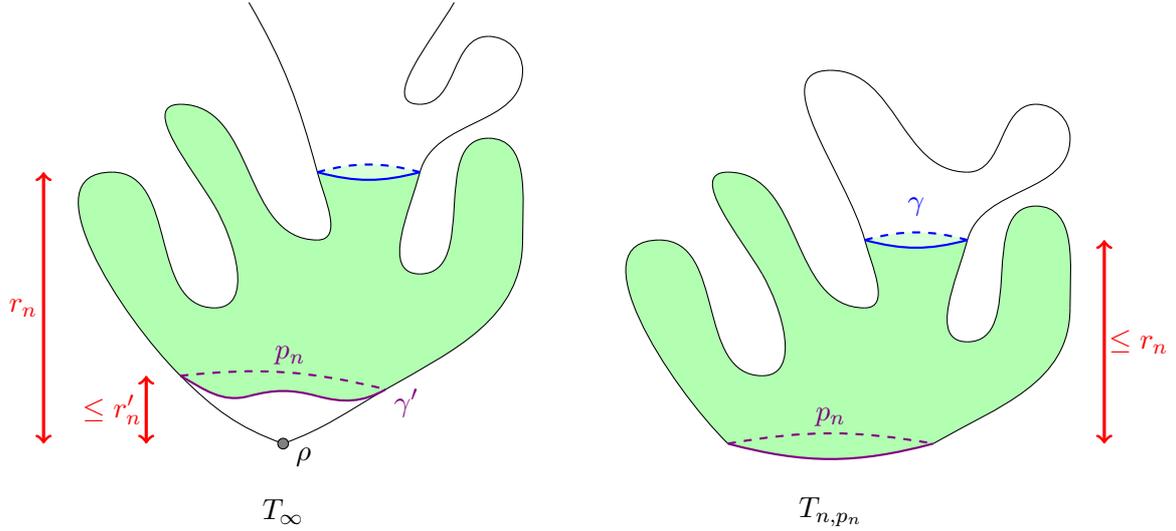
\begin{figure}
\begin{center}
\begin{tikzpicture}[scale=0.9]
\fill[green!15] (0.5,4) to[bend right =15] (2,4) to[bend right =15] (0.5,4);
\fill[green!30] (-1.5,1) to[out=330, in=195] (-0.5,0.7) to[out=15, in=165] (0.5,0.7) to[out=345, in=210] (1.5,0.8) to[out=30, in=270] (3.5,3) to[out=90, in=0] (3,4.5) to[out=180, in=0] (2,2.5) to[out=180, in=255] (2,4) to[bend left =15] (0.5,4) to[in=0, out=285] (0.5,3) to[in=0, out=180] (-1.5,5) to[in=120, out=180] (-1,3.5) to[in=0, out=300] (-1,2) to[in=0, out=180] (-2.5,4) to[in=135, out=180] (-1.5,1);

\draw(0,0) to[out=160, in=315] (-1.5,1);
\draw(0,0)node{} to[out=20, in=210] (1.5,0.8);
\draw(-1.5,1) to[out=135, in=180] (-2.5,4);
\draw(-2.5,4) to[out=0, in=180] (-1,2);
\draw(-1,2) to[out=0, in=300] (-1,3.5);
\draw(-1,3.5) to[out=120, in=180] (-1.5,5);
\draw(-1.5,5) to[out=0, in=180] (0.5,3);
\draw(1.5,0.8) to[out=30, in=270] (3.5,3);
\draw(3.5,3) to[out=90, in=0] (3,4.5);
\draw(3,4.5) to[out=180, in=0] (2,2.5);
\draw(0.5,3) to[out=0, in=285] (0.5,4);
\draw(2,2.5) to[out=180, in=255] (2,4);

\draw(0.5,4) to[out=105, in=300] (-0.5,6.5);
\draw(2,4) to[out=75, in=270] (3.5,5.5);
\draw(3.5,5.5) to[out=90, in=0] (3,6);
\draw(3,6) to[out=180, in=0] (2,5);
\draw(2,5) to[out=180, in=240] (2.5,6.5);

\draw[violet, thick] (-1.5,1) to[out=330, in=195] (-0.5,0.7);
\draw[violet, thick] (-0.5,0.7) to[out=15, in=165] (0.5,0.7);
\draw[violet, thick] (0.5,0.7) to[out=345, in=210] (1.5,0.8);
\draw[violet, dashed, thick] (-1.5,1) to[bend left =10] (1.5,0.8);
\draw[blue, thick] (0.5,4) to[bend right =15] (2,4);
\draw[blue, dashed, thick] (0.5,4) to[bend left =15] (2,4);

\draw[<->, red, very thick] (-2,0)--(-2,1);
\draw[red] (-2.5,0.5) node[texte]{$\leq r'_n$};
\draw (0.3,-0.2) node[texte]{$\rho$};
\draw[violet] (1.8,0.6) node[texte]{$\gamma'$};
\draw[violet] (0.1,1.3) node[texte]{$p_n$};
\draw[<->, red, very thick] (-3.5,0)--(-3.5,4);
\draw[red] (-3.8,2) node[texte]{$r_n$};
\draw (0,-1) node[texte]{$T_{\infty}$};

\begin{scope}[shift={(8,-1)}]
\fill[green!15] (0.5,4) to[bend right =15] (2,4) to[bend right =15] (0.5,4);
\fill[green!30] (-1.5,1) to[bend right =15] (1.5,1) to[out=30, in=270] (3.5,3) to[out=90, in=0] (3,4.5) to[out=180, in=0] (2,2.5) to[out=180, in=255] (2,4) to[bend left =15] (0.5,4) to[in=0, out=285] (0.5,3) to[in=0, out=180] (-1.5,5) to[in=120, out=180] (-1,3.5) to[in=0, out=300] (-1,2) to[in=0, out=180] (-2.5,4) to[in=135, out=180] (-1.5,1);

\draw(-1.5,1) to[out=135, in=180] (-2.5,4);
\draw(-2.5,4) to[out=0, in=180] (-1,2);
\draw(-1,2) to[out=0, in=300] (-1,3.5);
\draw(-1,3.5) to[out=120, in=180] (-1.5,5);
\draw(-1.5,5) to[out=0, in=180] (0.5,3);
\draw(1.5,1) to[out=30, in=270] (3.5,3);
\draw(3.5,3) to[out=90, in=0] (3,4.5);
\draw(3,4.5) to[out=180, in=0] (2,2.5);
\draw(0.5,3) to[out=0, in=285] (0.5,4);
\draw(2,2.5) to[out=180, in=255] (2,4);
\draw(2,4) to[out=75, in=270] (3.5,5.5);
\draw(3.5,5.5) to[out=90, in=0] (3,6);
\draw(3,6) to[out=180, in=0] (2,5);
\draw(0.5,4) to[out=105, in=180] (0,6.5);
\draw(0,6.5) to[out=0, in=180] (2,5);

\draw[violet, thick] (-1.5,1) to[bend right =15] (1.5,1);
\draw[violet, dashed, thick] (-1.5,1) to[bend left =10] (1.5,1);
\draw[blue, thick] (0.5,4) to[bend right =15] (2,4);
\draw[blue, dashed, thick] (0.5,4) to[bend left =15] (2,4);

\draw[<->, red, very thick] (4,1)--(4,4);
\draw[red] (4.5,2.5) node[texte]{$\leq r_n$};
\draw[violet] (0,1.4) node[texte]{$p_n$};
\draw[blue] (1.25,4.5) node[texte]{$\gamma$};
\draw (0,0) node[texte]{$T_{n,p_n}$};
\end{scope}
\end{tikzpicture}
\end{center}
\vspace{-5mm}
\caption{Illustration of Lemma \ref{couplagebord}. With high probability, there are two cycles $\gamma'$ and $\gamma$ such that the two green parts coincide.} \label{figurecoupling}
\end{figure}

\begin{proof}
We start by describing a coupling between the UIPT and the UIPT with a boundary of length $p_n$, that we write $T_{\infty,p_n}$. We consider the peeling by layers $\mathscr{L}$ of the UIPT (see section 4.1 of \cite{CLGpeeling}) and we write $\tau_{p_n}$ for the first time at which the perimeter of the discovered region is equal to $p_n$ (note that this time is always finite since the perimeter can increase by at most $1$ at each peeling step). By the spatial Markov property of the UIPT, the part that is still unknown at time $\tau_{p_n}$ has the distribution of $T_{\infty,p_n}$. Moreover, by the results of Curien and Le Gall (Theorem 1 of \cite{CLGpeeling}), since $p_n=o(r_n^2)$, we have $\tau_{p_n}=o(r_n^3)$. By using Proposition 9 of \cite{CLGpeeling} (more precisely the convergence of $H$), we obtain that the smallest hull of $T_{\infty}$ containing $t_{\tau_{p_n}}^{\mathscr{L}} \left( T_{\infty} \right)$ has radius $o(r_n)$ in probability. Hence, our result holds if we replace $T_{n, p_n}$ by $T_{\infty, p_n}$.

Hence, it is enough to prove that there are couplings between $T_{\infty,p_n}$ and $T_{n,p_n}$ such that 
\[ \mathbb{P} \left(  B^{\bullet}_{r_n}(T_{n,p_n})=B^{\bullet}_{r_n}(T_{\infty,p_n}) \right) \xrightarrow[n \to +\infty]{} 1.\]
The proof relies on asymptotic enumeration results and is essentially the same as that of Proposition 12 of \cite{HBP}: by using the above coupling of $T_{\infty, p_n}$ and $T_{\infty}$ we can show that
\[ \left( \frac{1}{\sqrt{n}} |\partial B^{\bullet}_{r_n}(T_{\infty, p_n})|, \frac{1}{n} |B^{\bullet}_{r_n}(T_{\infty, p_n})| \right) \xrightarrow[n \to +\infty]{(P)} (0,0). \]
Moreover, if $q_n=o(\sqrt{n})$ and $v_n=o(n)$ and if $t_n$ is a triangulation with two holes of perimeters $p_n$ and $q_n$ (rooted on the boundary of the $p_n$-gon) and $v_n$ vertices that is a possible value of $B^{\bullet}_{r_n}(T_{\infty, p_n})$ for all $n \geq 0$, then
\[ \frac{\mathbb{P} \left( B^{\bullet}_{r_n}(T_{n, p_n})=t_n \right)}{\mathbb{P} \left( B^{\bullet}_{r_n}(T_{\infty, p_n})=t_n \right)} \xrightarrow [n \to +\infty]{} 1\]
by the enumeration results, and we can conclude as in Proposition 12 of \cite{HBP}.
\end{proof}

We will also need another coupling lemma where we do not compare hulls of a fixed radius, but rather the parts of triangulations that have been discover after a fixed number of peeling steps.

\begin{lem}
Let $j_n=o(n^{3/4})$, and let $\mathscr{A}$ be a peeling algorithm. Then there are couplings between $T_{n}$ and $T_{\infty}$ such that
\[\mathbb{P} \left( t_{j_n}^{\mathscr{A}}(T_{n})=t_{j_n}^{\mathscr{A}}(T_{\infty}) \right) \xrightarrow[n \to +\infty]{} 1.\]
\end{lem}

\begin{proof}
We write $P_{\infty}(j)$ and $V_{\infty}(j)$ for respectively the perimeter and volume of $t_{j}^{\mathscr{A}}(T_{\infty})$.
By the results of \cite{CLGpeeling} we have the convergences
\begin{equation}\label{PVinftysmall}
\frac{1}{\sqrt{n}} \sup_{0 \leq j \leq j_n} P_{\infty}(j) \xrightarrow[n \to +\infty]{} 0 \hspace{5mm} \mbox{and} \hspace{5mm} \frac{1}{n} \sup_{0 \leq j \leq j_n} V_{\infty}(j) \xrightarrow[n \to +\infty]{} 0
\end{equation}
in probability, so there are $p_n=o(\sqrt{n})$ and $v_n=o(n)$ such that 
\[\mathbb{P} \left( P_{\infty}(j_n) \leq p_n \mbox{ and } V_{\infty}(j_n) \leq v_n \right) \to 1.\]
But by the enumeration results \eqref{enumeration}, \eqref{enumerationSphere} and by \eqref{UIPT}, if $t_n$ is a rooted triangulation with perimeter at most $p_n$ and volume at most $v_n$, we have
\[\frac{\P \left(t_{j_n}^{\mathscr{A}} (T_n)=t_n \right)}{\P \left(t_{j_n}^{\mathscr{A}} (T_{\infty})=t_n \right)} = \frac{\mathbb{P} \big( t_n \subset T_{n} \big)}{\mathbb{P} \big( t_n \subset T_{\infty} \big)} \xrightarrow[n \to +\infty]{} 1.\]
As in Proposition 12 of \cite{HBP}, this proves that the total variation distance between the distributions of $t_{j_n}^{\mathscr{A}}(T_{n})$ and $t_{j_n}^{\mathscr{A}}(T_{\infty})$ goes to $0$ as $n \to +\infty$, which proves our claim and the lemma.
\end{proof}

By combining this last lemma and the estimates \eqref{PVinftysmall}, we immediately obtain estimates about the peeling process on finite uniform triangulations. We write $P_n(j)$ and $V_n(j)$ for the perimeter and volume of $t^{\mathscr{A}}_j(T_n)$.

\begin{corr} \label{estimatesPV}
Let $j_n=o(n^{3/4})$. Then we have the following convergences in probability:
\[ \frac{1}{\sqrt{n}} \sup_{0 \leq j \leq j_n} P_n(j) \xrightarrow[n \to +\infty]{} 0 \hspace{5mm} \mbox{and} \hspace{5mm} \frac{1}{n} \sup_{0 \leq j \leq j_n} V_n(j) \xrightarrow[n \to +\infty]{} 0.\]
\end{corr}

Finally, we show a result about small cycles surrounding the boundary in uniform triangulations with a perimeter small enough compared to their volume.

\begin{lem}\label{smallcycle}
Let $p_n=o(\sqrt{n})$ and $r_n=o(n^{1/4})$ be such that $p_n=o(r_n^2)$. Then for all $\eps>0$, the probability of the event
\begin{center}
"there is a cycle $\gamma$ in $T_{n,p_n}$ of length at most $r_n$ such that the part of $T_{n,p_n}$ lying between $\partial T_{n,p_n}$ and $\gamma$ contains at most $\eps n$ vertices"
\end{center}
goes to $1$ as $n \to +\infty$.
\end{lem}

This result is not surprising. In the context of quadrangulations with a non-simple boundary, it is a consequence of the convergence of quadrangulations with boundaries to Brownian disks, see \cite{BM15}. However, no scaling limit result is known yet for triangulations with boundaries. Hence, we will rely on a result of Krikun about small cycles in the UIPT, that we will combine with Lemma \ref{couplagebord}. Here is a restatement of Theorem 6 of \cite{Kri04}.

\begin{thm}[Krikun] \label{Krikuncycle}
For all $\eps>0$, there is a constant $C$ such that for all $r$, with probability at least $1-\eps$ there is a cycle of length at most $Cr$ surrounding $B_r^{\bullet}(T_{\infty})$ and lying in $B_{2r}^{\bullet}(T_{\infty})$.
\end{thm}

Note that Krikun deals with type-II triangulations, i.e. with multiple edges but no loops, but the decomposition used in \cite{Kri04} is still valid and even a bit simpler in the type-I setting, see \cite{CLGmodif}. The fact that the cycle stays in $B_{2r}^{\bullet}(T_{\infty})$ is not in the statement of the theorem in \cite{Kri04} but it is immediate from its proof.

\begin{proof}[Proof of Lemma \ref{smallcycle}]
By Lemma \ref{couplagebord} it is possible to couple $T_{\infty}$ and $T_{n,p_n}$ in such a way that
\begin{equation} \label{finalcoupling}
\mathbb{P} \left( B^{\bullet}_{r_n} (T_{\infty}) \backslash B^{\bullet}_{r'_n} (T_{\infty}) \subset B^{\bullet}_{r_n} \left( T_{n,p_n} \right) \right) \xrightarrow[n \to +\infty]{} 1,
\end{equation}
where $r'_n=o(r_n)$. On the other hand, by Theorem \ref{Krikuncycle}, we have
\[ \mathbb{P} \left( \mbox{there is a cycle $\gamma$ of length $\leq r_n$ in $B_{2r'_n}^{\bullet}(T_{\infty})$ that surrounds $B^{\bullet}_{r'_n} \left( T_{\infty} \right)$} \right) \xrightarrow[n \to +\infty]{} 1.\]
For $n$ large enough we have $r_n \geq 2r'_n$ so if such a $\gamma$ exists in then it must stay in $B^{\bullet}_{r_n} \left( T_{\infty} \right)$. Since $r_n=o(n^{1/4})$, the probability that the number of vertices lying inside of $\gamma$ is greater than $\eps n$ goes to $0$ by Theorem 2 of \cite{CLGpeeling}.
But if the event of \eqref{finalcoupling} holds and if such a cycle exists in $T_{\infty}$, then in $T_{n,p_n}$ there is a cycle $\gamma$ of length at most $r_n$ such that the part of $T_{n,p_n}$ lying between $\partial T_{n,p_n}$ and $\gamma$ contains at most $\eps n$ vertices.
\end{proof}

\section{Proof of Theorem \ref{mainthm}}

Our main task will be to prove the following proposition.

\begin{prop} \label{propcycle}
Let $k_n=o(n^{5/4})$. Then there are $t_n \in \mathscr{T}_n$ and $\ell_n=o(n^{1/4})$ such that conditionally on $T_n(0)=t_n$, the probability that there is a cycle of length at most $\ell_n$ that separates $T_n(k_n)$ in two parts of volume at least $\frac{n}{4}$ goes to $1$ as $n \to +\infty$.
\end{prop}

We first define the initial triangulation $T_n(0)$ we will be interested in: let $T^1_n(0)$ and $T^2_n(0)$ be two independent uniform triangulations with a boundary of length $1$ and with respectively $\lfloor \frac{n-1}{2} \rfloor$ and $\lceil \frac{n-1}{2} \rceil$ inner vertices. We write $T_n(0)$ for the triangulation obtained by gluing together the boundaries of $T^1_n(0)$ and $T^2_n(0)$.

We will now perform an exploration of the triangulation while it gets flipped: the part $T^1_n$ will be considered as the "discovered" part and $T^2_n$ as the "unknown" part of the map. More precisely, we define by induction $T^1_n(k)$ and $T^2_n(k)$ such that $T_n(k)$ is obtained by gluing together the boundaries of $T^1_n(k)$ and $T^2_n(k)$. The two triangulations for $k=0$ are defined above. Now assume we have constructed $T^1_n(k)$ and $T^2_n(k)$. Then:
\begin{itemize}
\item
if $e_k$ lies inside of $T^1_n(k)$ then $T^1_n(k+1)=\mathfrak{flip}(T^1_n(k), e_k)$ and $T^2_n(k+1)=T^2_n(k)$,
\item
if $e_k$ lies inside of $T^2_n(k)$ then $T^1_n(k+1)=T^1_n(k)$ and $T^2_n(k+1)=\mathfrak{flip}(T^2_n(k), e_k)$,
\item
if $e_k \in \partial T^1_n(k)$, we write $f_k$ for the face of $T^2_n(k)$ that is adjacent to $e_k$, and we let $T^2_n(k+1)$ be the connected component of $T^2_n(k) \backslash f_k$ with the largest volume and $T^1_n(k+1)=\mathfrak{flip} \big( T_n(k) \backslash T^2_n(k+1), e_k \big)$.

\end{itemize}

We now set $\widetilde{P}_n(k)=|\partial T^1_n(k)|$ and $\widetilde{V}_n(k)= \left| V \left( T^1_n(k) \right) \right|- \left| V \left( T^1_n(0) \right) \right|+1$. Note that $\widetilde{V}_n(k)$ is nondecreasing in $k$.

For $k \geq 0$, we define a random variable $e_k^* \in E(T^1_n(k)) \cup \{ \star \}$, where $\star$ is an additional state corresponding to all the edges not in $E(T^1_n(k))$, as follows: if $e_k$ lies inside or on the boundary of $T^1_n(k)$ then $e^*_k=e_k$, and if not then $e_k^*=\star$. We also define $\mathcal{F}_k$ as the $\sigma$-algebra generated by the variables $\left( T^1_n(i) \right)_{0 \leq i \leq k}$ and $(e_i^*)_{0 \leq i \leq k-1}$.

\begin{lem} \label{resteruniforme}
For all $k$, conditionally on $\mathcal{F}_k$, the triangulation $T^2_n(k)$ is a uniform triangulation with a boundary of length $\widetilde{P}_n(k)$ and $\lceil \frac{n+1}{2} \rceil -\widetilde{V}_n(k)$ inner vertices.
\end{lem}

\begin{proof}
We prove the lemma by induction on $k$. For $k=0$ it is obvious by the definition of $T^2_n(0)$. Let $k \geq 0$ be such that the lemma holds for $k$.
\begin{itemize}
\item[$\bullet$]
If $e_k^*$ lies inside $T^1_n(k)$, the result follows from the fact that $T^2_n(k)=T^2_n(k+1)$ and that conditionally on $\mathcal{F}_k$, the triangulation $T^2_n(k)$ is independent of $e^*_k$.
\item[$\bullet$]
If $e_k^*=\star$, it follows from the invariance of the uniform measure on $\mathscr{T}_{n,p}$ under flipping of a uniform edge among those which do not lie on the boundary.
%
%
\item[$\bullet$]
If $e_k^* \in \partial T^1_n(k)$, this is a standard peeling step: by invariance under rerooting of a uniform triangulation with fixed perimeter and volume, conditionally on $\mathcal{F}_k$ and $e_k$, the triangulation $T^2_n(k)$ rooted at $e_k$ is uniform. Hence, if the third vertex of the face $f_k$ of $T^2_n(k)$ adjacent to $e_k$ lies inside of $T^2_n(k)$, the remaining part of $T^2_n(k)$ is a uniform triangulation with a boundary of length $\widetilde{P}_n(k)+1$ and $\lceil \frac{n+1}{2} \rceil-\widetilde{V}_n(k)-1$ inner vertices. If the third vertex of $f_k$ lies on $\partial T^2_n(k)$, then the face $f_k$ separates $T^2_n(k)$ in two independent uniform triangulations with fixed perimeters and volumes, and the lemma follows.
\end{itemize}

%
\end{proof}

We now define the stopping times $\tau_j$ as the times at which the flipped edge lies on the boundary of the unknown part of the map, that is, the times $k$ at which we discover new parts of $T^2_n(k)$: we set $\tau_0=0$ and $\tau_{j+1}=\inf \{ k>\tau_j | e_k \in \partial T^1_n(k)\}$ for $j \geq 0$. We also write $P_n(j)=\widetilde{P}_n(\tau_j+1)$ and $V_n(j)=\widetilde{V}_n(\tau_j+1)$.

Then Lemma \ref{resteruniforme} shows that $\left( P_n, V_n \right)$ is a Markov chain with the same transitions as the perimeter and volume processes associated to the peeling process of a uniform triangulation with a boundary of length $1$ and $\lceil \frac{n-1}{2} \rceil$ inner vertices. Hence, Corollary \ref{estimatesPV} provides estimates for this process. Our next lemma will allow us to estimate the times $\tau_j$.

\begin{lem} \label{estimateTau}
Let $k_n=o(n^{5/4})$. Then for all $\eps>0$ we have
\[\mathbb{P} \left( \tau_{\eps n^{3/4}} > k_n \right) \xrightarrow[n \to +\infty]{} 1.\]
\end{lem}

\begin{proof}
Conditionally on $P_n$, the variables $\tau_{j+1}-\tau_j$ are independent geometric variables with respective parameters $\frac{P_n(j)}{n}$. Hence, $\tau_{\eps n^{3/4}}$ dominates the sum $S_n$ of $\eps n^{3/4}$ i.i.d. geometric variables with parameter $Q_n=\frac{1}{n} \max_{0 \leq j \leq \eps n^{3/4}} P_n(j)$. We have
\[\E [S_n|P_n]=\eps n^{3/4} Q_n=\eps n^{5/4} \times \frac{1}{\sqrt{n}} \max_{0 \leq j \leq \eps n^{3/4}} P_n(j).\]
By the results of \cite{CLGpeeling}, the factor $\frac{1}{\sqrt{n}} \max_{0 \leq j \leq \eps n^{3/4}} P_n(j)$ converges in distribution, so $\frac{\E [S_n|P_n]}{\eps n^{5/4}}$ converges in distribution so $\frac{ \mathbb{E} \left[ S_n | P_n \right]}{k_n} \to +\infty$ in probability. By the weak law of large numbers we get $\frac{S_n}{k_n} \to +\infty$ in probability so $\frac{\tau_{\eps n^{3/4}}}{k_n} \to +\infty$ in probability.
\end{proof}

By combining Corollary \ref{estimatesPV} and Lemma \ref{estimateTau} we get the following result.

\begin{lem} \label{estimateTilde}
Let $k_n=o(n^{5/4})$. Then we have the convergences
\[ \frac{1}{\sqrt{n}} \widetilde{P}_n(k_n) \xrightarrow[n \to +\infty]{} 0 \hspace{5mm} \mbox{and} \hspace{5mm} \frac{1}{n} \widetilde{V}_n(k_n) \xrightarrow[n \to +\infty]{} 0\]
in probability.
\end{lem}

\begin{proof}
By Lemma \ref{estimateTau} there is a deterministic sequence $j_n=o(n^{3/4})$ such that $\mathbb{P} \left( \tau_{j_n} > k_n \right) \to 1$. This means that with probability going to $1$ as $n \to +\infty$ there is $J \leq j_n$ such that $\tau_J < k_n \leq \tau_{J+1}$ so \[\widetilde{P}_n(k_n)=P_n(J) \leq \sup_{0 \leq j \leq j_n} P_n(j) \hspace{5mm} \mbox{and} \hspace{5mm} \widetilde{V}_n(k_n)=V_n(J) \leq \sup_{0 \leq j \leq j_n} V_n(j).\]
But we know from Corollary \ref{estimatesPV} that
\[ \left( \frac{1}{\sqrt{n}} \sup_{0 \leq j \leq j_n} P_n(j), \frac{1}{n} \sup_{0 \leq j \leq j_n} V_n(j) \right) \xrightarrow[n \to +\infty]{(P)} 0,\]
which proves Lemma \ref{estimateTilde}.
\end{proof}

So $T_n^2(k_n)$ has the distribution of $T_{n/2-\widetilde{V}_n(k_n),\widetilde{P}_n(k_n)}$ and there is $p_n=o(\sqrt{n})$ such that
\[ \P \left( \widetilde{P}_n(k_n)<p_n \mbox{ and } n/2-\widetilde{V}_n(k_n)>\frac{n}{3} \right) \xrightarrow[n \to +\infty]{} 1.\]
Let $r_n$ be such that $r_n=o(n^{1/4})$ and $p_n=o(r_n^2)$ (take for example $r_n=n^{1/8} p_n^{1/4}$). By Lemma \ref{smallcycle}, with probability going to $1$ as $n \to +\infty$, there is a cycle $\gamma$ in $T_n^2(k_n)$ of length at most $r_n$ such that the part of $T^2_n(k_n)$ lying between $\partial T^2_n(k_n)$ and $\gamma$ has volume at most $\frac{n}{6}$. Moreover we have $\widetilde{V}_n(k_n)=o(n)$ in probability by Lemma \ref{estimateTilde}, so the two parts of $T_n(k_n)$ separated by $\gamma$ both have volume at least $\frac{n}{4}$, which proves Proposition \ref{propcycle}.

The proof of our main theorem is now easy: let $\mathscr{T}_n^{\bowtie}$ be the set of the triangulations $t$ of the sphere with $n$ vertices in which there is a cycle of length at most $\ell_n$ that separates $t$ in two parts of volume at least $\frac{n}{4}$. Let also $k_n=o(n^{5/4})$. By Proposition \ref{propcycle} we have
\[ \mathbb{P} \left( T_n(k_n) \in \mathscr{T}_n^{\bowtie} \right) \xrightarrow[n \to +\infty]{} 1, \]
whereas by Corollary 1.2 of \cite{LGP08}, if $T_n(\infty)$ denotes a uniform variable on $\mathscr{T}_n$ we have
\[ \mathbb{P} \left( T_n(\infty) \in \mathscr{T}_n^{\bowtie} \right) \xrightarrow[n \to +\infty]{} 0. \]
Hence, the total variation distance between the distributions of $T_n(k_n)$ and $T_n(\infty)$ goes to $1$ as $n \to +\infty$ so the mixing time is greater than $k_n$ for $n$ large enough. Since this is true for any $k_n=o(n^{5/4})$, the mixing time must be at least $c n^{5/4}$ with $c>0$.

We end this paper by a few remarks about our lower bound and an open question.

\begin{rem}
We proved a lower bound on the mixing time in the worst case, but our proof still holds for the mixing time from a typical starting point. We just need to fix $\eps>0$ small, take as initial condition a uniform triangulation $T_n(0)$ conditioned on $\left|\partial B^{\bullet}_{n^{1/4}}(T_n(0)) \right| \leq \eps \sqrt{n}$ and $\frac{n}{3} \leq \left| B^{\bullet}_{n^{1/4}}(T_n(0)) \right| \leq \frac{2n}{3}$ and let $T_n^1(0)=B^{\bullet}_{n^{1/4}}(T_n(0))$. The event on which we condition has probability bounded away from $0$ (by the results of \cite{CLGpeeling} and coupling arguments) and after time $o(n^{5/4})$ there is still a seperating cycle of length $O(\eps^{1/2} n^{1/4})$.
\end{rem}

\begin{rem}
Here is a back-of-the-enveloppe computation that leads us to believe the lower bound we give is sharp if we start from a typical triangulation. The lengths of the geodesics in a uniform triangulation of volume $n$ are of order $n^{1/4}$, so if we fix two vertices $x$ and $y$ the probability that a flip hits the geodesic from $x$ to $y$ is roughly $n^{-3/4}$. Hence, if we do $n^{5/4}$ flips, about $n^{1/2}$ of them will affect the distance between $x$ and $y$. If we believe that this distance evolves roughly like a random walk, it will vary of about $\sqrt{n^{1/2}}=n^{1/4}$, which shows we are at the right scale. Of course, there are many reasons why this computation seems hard to be made rigourous, but it does not seem to be contradicted by numerical simulations.
\end{rem}

Finally, note that even in the simpler case of triangulations of a polygon, the lower bound $n^{3/2}$ is believed to be sharp but the best known upper bound \cite{MT97} is only $n^{5+o(1)}$. In our case we were not even able to prove the following.

\begin{conj}
The mixing time of $(T_n(k))_{k \geq 0}$ is polynomial in $n$.
\end{conj}

\appendix

\section{Connectedness of the flip graph for type-I triangulations}

In this appendix, we show that the Markov chain we study is indeed irreducible.

\begin{lem}\label{irreducibility}
Let $\mathscr{G}_n$ be the graph whose vertex set is $\mathscr{T}_n$ and where two triangulations are related if one can pass from one to the other by a flip. Then $\mathscr{G}_n$ is connected and its diameter is linear in $n$.
\end{lem}

\begin{proof}
It is proved in \cite{W36} that the flip graph for type-III triangulations is connected, and in \cite{K97} that its diameter is linear in $n$. Hence, it is enough to show that any triangulation is connected to a type-III triangulation in $\mathscr{G}_n$ by a linear number of edges. If $t$ is a finite triangulation with loops, it contains a \textit{minimal} loop, that is, a loop dividing the sphere in two parts, one of which contains no loop. By flipping a minimal loop we delete a loop without to create any new one, so we make the number of loops decrease and we can delete all loops in a linear number of flips. Moreover, if $t$ contains no loop and there are two edges $e_1,e_2$ between the same pair of vertices, then flipping $e_1$ does not create any loop or additional multiple edges, so we can also delete all multiple edges in a linear number of flips.
\end{proof}


\bibliographystyle{abbrv}
\bibliography{bibli}

\end{document}